\documentclass[preprint,12pt]{elsarticle}

\usepackage[latin1]{inputenc}

\usepackage{enumerate} 
\usepackage{amsthm}
\usepackage{amsmath}
\usepackage{amssymb}
\usepackage{calrsfs}
\usepackage{graphicx}
\DeclareGraphicsExtensions{.png,.pdf,.jpg}
\usepackage{amssymb}
\newtheorem{remark}{Remark}
\newtheorem{theorem}{Theorem}[section]
\newtheorem{lemma}{Lemma}[section]

\newtheorem{definition}{Definition}[section]
\newtheorem{proposition}{Proposition}[section]
 
\newcommand{\im}{{\rm{Im}} }

\journal{Applied and  Computational Harmonic Analysis}

\begin{document}

\begin{frontmatter}


\fntext[Nota1]{The first author was partially supported by  MinCyT, ANPCyT PICT2014-1480 and UBA, UBACyT 20020170100430BA.}
\fntext[Nota2]{The second author was partially supported by PIO CONICET-UNGS 144-20140100011-CO.}

\title{A generalized version of the 2-microlocal frontier prescription}

\author[label1]{Ursula Molter \fnref{Nota1}}
\author[label2]{Mariel Rosenblatt \fnref{Nota2}}
\address[label1]{FCEyN, Universidad de Buenos Aires, and IMAS, CONICET-UBA, Argentina}
\address[label2]{IDH, Universidad Nacional de General Sarmiento, Buenos Aires, Argentina}
\ead{umolter@dm.uba.ar (Ursula Molter), mrosen@campus.ungs.edu.ar}

\begin{abstract}
The characterization of  local regularity is a fundamental issue in signal and image processing, since it contains relevant information about the underlying systems. The 2-microlocal frontier, a monotone concave downward curve in $\mathbb {R}^2$,  provides a complete and profound classification of  pointwise singularity.

In \cite{Meyer1998}, \cite{GuiJaffardLevy1998} and \cite{LevySeuret2004} the authors show the following: given a monotone concave downward curve in the plane it is possible to exhibit one function (or distribution) such that its 2-microlocal frontier al $x_0$ is the given curve.

In this work we are able to unify the previous results, by obtaining a large class of functions (or distributions), that includes the three examples mentioned above, for which the 2-microlocal frontier is the given curve. The three examples above are in this class.

Further, if the curve is a line, we characterize all the functions
whose 2-microlocal frontier at $x_0$ is the given line.

\end{abstract}

\end{frontmatter}


\section{Introduction}
\label{Intro}

Singularities of functions or signals are points at which the function lacks regularity. The detection and characterization of singularities is an important topic in signal processing, since they contain significant information about the phenomena.

There are several types of singularities that can be illustrated in known examples, e.g. the function $f(x)=|x-x_{0}|^{\alpha}$,  $ 0 < \alpha < 1$, has a  {\em cusp} type singularity at $x_0$, which is non oscillating .

In contrast, the function  $ f(x)=|x-x_{0}|^{\alpha} \sin  \left(  |x-x_{0}|^{-\beta} \right) $, with $ 0 < \alpha < 1$ and   $\beta> 0 $, has a {\em chirp} type singularity at $ x_0 $, with an oscillatory behaviour around $ x_0 $.

{\em Oscillating} and {\em non oscillating} are a first and clear distinction among singularities. However, the intuitive notion of {\em oscillation} is not enough to characterize more complex structures.
For example the function
$ f(x)=|x-x_{0}| \sin  \left(  |x-x_{0}|^{-1} \right) +|x-x_{0}|^{3/2} $ is oscillating at $x_0$ but does not have a chirp type singularity.

From a mathematical point of view, it is important to characterize the different singularities. Classical functions such as the  Weierstrass function \cite{Daoudi1998}, Riemann functions and other examples \cite{JaffardMeyer1996, Jaffard1997} present cusp or oscillating type singularities at almost every point.

Likewise, the relevance of characterizing singularities is also important in applications, because it is fundamental to accurately describe natural and social phenomena.
In fact, many signals from natural and social phenomena (EEG, ECG, data from financial markets, among others) usually present cusp type singularities. But it is also possible to detect oscillating singularities in signals from physical and natural phenomena, for example, echolocation waves emitted by bats \cite{Kopsinis_echobat_2010} and  hydrodynamic turbulence phenomena \cite{Abry2011} present oscillating singularities. Recently, it has been confirmed that these type of oscillating structures appear at  gravitational waves, recorded by the laser interferometry gravitational-wave observatories LIGO (located in the USA) and Virgo (located in Europe),  during the merger of two black holes in 2015 \cite{AbbottLIGO2015} and from a binary neutron star inspiral in 2017 \cite {AbbottLIGO2017}.

The complete characterization of a pointwise singularity requires several parameters. One of the most commonly used parameter in signal processing is the {\em pointwise  H\"older exponent}. However, the information provided by this exponent is insufficient to distinguish oscillating singularities from non-oscillating ones. To complement this information, other parameters have been proposed: {\em the local H\"older exponent} \cite {SeuretLevy2002}, the {\em oscillation and chirp} exponents \cite{JaffardMeyer1996}, \cite{Arneodo1998} and the { \em weak scaling } exponent \cite {Meyer1998}. Recently, a novel quantification of local regularity based on $p$-exponents, for $p>0$, has been presented in \cite{Jaffard2016}.

Under some global regularity assumptions of the function  $f$, the classical regularity exponents can all be extracted from a concave downward curve in $\mathbb {R}^2$: the {\em {2-microlocal frontier}} at $x_0$. This curve is defined by means of the 2-microlocal spaces ${C}_{x_{0}}^{s,s'}$, with the  parameters $s,s'\in\mathbb{R}$.

2-microlocal spaces were introduced by J.M Bony \cite{Bony1984} to examine the propagations  of singularities for semi-linear hyperbolic equations. They are defined as functional spaces embedded in the space of tempered distributions $\mathcal{S}'(\mathbb {R}) $ and their fundamental property  is that they are stable under the action of differential and integral operators \cite {Bony1984}, that is  
$$ f \in {C}_{x_{0}}^{s,s'} \quad \Longleftrightarrow \quad f^{(n)} \in {C}_{x_{0}}^{s-n,s'} \;\;\; \forall n \in \mathbb{N}.$$

The original definition of  ${C}_{x_{0}}^{s,s'}$ is associated to conditions on Littlewood-Paley decompositions of tempered distributions. In \cite {Jaffard1991}, S. Jaffard reformulates these conditions by means of the wavelet transform, providing another characterization of the 2-microlocal spaces of J.M Bony.

In this work we will use the wavelet approach to define these spaces.
Recall that a function $\psi\in L^2(\mathbb{R})$ is called a \textit{wavelet} if $\{\psi_{j,k}= 2^{j/2} \psi(2^{j} x - k)\}_j,k\in\mathbb{Z}$ forms an orthonormal basis of  $L^2(\mathbb{R})$. In that case the wavelet coefficients of $f\in L^2(\mathbb{R})$ are

\begin{equation}
c_{j,k} = 2^{j/2}\langle f,\psi_{j,k}\rangle \;\; \text{ with } \; \; \psi_{j,k}= 2^{j/2} \psi(2^{j} x - k).
\end{equation}

To extend the wavelet expansion to the space of tempered distributions we will require that the {\em mother wavelet} belongs to the space of Schwartz $ \mathcal {S} (\mathbb {R}) $. In fact it is not necessary to have $\psi \in \mathcal {S} (\mathbb {R}) $ when looking at the local behaviour of $f$. It is enough that $\psi$ has sufficient vanishing moments and its first derivatives are of fast decay.

We then have the following definition  (see  \cite{Jaffard1991} and the references cited therein for details about this statement),

\begin{definition}
	For a tempered distribution $f$ we say that 
	
	\begin{center}
		$f \in {C}_{x_{0}}^{s,s'}$ if and only if there exits $C > 0$ such that $$|c_{j,k}| \leq C 2^{-j s}\;(1 + | k - 2^{j} x_{0} | )^{-s'}$$
	\end{center}
	 for all $ j $ and $k \in {\mathbb Z}.$
\end{definition}

However, when analysing the local behaviour of $ f $ at $ x_0$  it is not necessary that $ f $ is defined at infinity. We will use the \textbf{local} $ {C}_{x_{0}}^ {s, s '} $ spaces, defined by  Y. Meyer and S. Jaffard \cite {JaffardMeyer1996, Meyer1998}. They are embedded in the space of distributions $ \mathcal{D}'(V) $, where $ V $ is an open set  containing $x_0$.

\begin{definition}
	Let  $V$ be an open neighbourhood of $x_0 $ and  $f\in \mathcal{D}'(V)$. We say that 
	$f$ belongs to the \textbf{local}  ${C}_{x_{0}}^{s,s'}  $if there exists an open neighbourhood of $x_0$ 
 $V_0\subsetneq V$  and  $F\in  \textbf{ global }{C}_{x_{0}}^{s,s'}$ such that
	$$f=F\;\; \text{in} \;\;V_0.$$
	
\end{definition}

The wavelet characterization of the  \textbf {local} $ {C}_{x_{0}}^{s, s'} $ spaces is summarized in the following equivalent definition. From now on we will denote with ${C}_{x_{0}}^{s,s'}$ the {\em local} 2-microlocal space.

\begin{definition}
	\label{pertenencia_2ml}
	$f\in {C}_{x_{0}}^{s,s'}$  if and only if there exists $C > 0$ such that
	\begin{eqnarray}
	|c_{j,k}| \leq C 2^{-j s}\;(1 + | k - 2^{j} x_{0} | )^{-s'}
	\end{eqnarray}
	for all $ j $ and $k \in {\mathbb Z}$ such that  $j\geq 0$ and $|\frac{ k}{2^j} - x_{0}| < 1$.
\end{definition}

\begin{remark}
	\label{wavelet_condiciones}
	If the mother wavelet $\psi$ satisfies that $\psi$ has  N vanishing moments,
	and that its first $r$ derivatives have fast decay,  then the  previous definition is valid if $ s, s' $ verify:

	\begin{equation}\label{condiciones_psi}
	r+s+\inf\ \left\{s',1\right\}>0 \text{ and } N>\sup\left\{s,s+s'\right\}.
	\end{equation}
	
In this work  we will consider $\psi$ to be in the Schwartz class for simplicity, for example the Meyer wavelet which has  all its vanishing moments and all its derivatives of fast decay. Consequently, the conditions required in (\ref {condiciones_psi}) are satisfied for all pairs $ (s, s') $.

	\end{remark}

It should be noted that for some  $(s,s')$ there also exist characterizations of local 2-microlocal spaces in time domain \cite{Kolwankar2002, Seuret2003,LevySeuret2004, Echelard2007}. Further, the notion of 2-microlocal regularity has been extended recently to the stochastic setting \cite{Herbin2009,BalancaHerbin2012,Balanca2014}.

In order to give a geometric description of the singular behaviour of a function,  in \cite {Meyer1998} Y. Meyer defines a convex set  $D(f,x_0)\in \mathbb{R}^2$  which is called the \textit{2-microlocal domain} of $ f $ at $ x_0 $:

\begin{equation}\label{dom2ml}
D(f,x_0)= \{(s,s'): f\in  {C}_{x_{0}}^{s,s'}\}.
\end{equation}

\begin{definition}
	\label{def_frontera}

	The boundary of the set $ D (f, x_0) $  is the 2-microlocal frontier of $ f $ at $ x_{0} $, which can be defined by the parametrization
	$$ s'\longmapsto \sup \{s : f \in {C}_{x_{0}}^{s,s'} \}.$$

	Denoting by $ \sigma = s + s' $, the 2-microlocal frontier of $ f $ at $ x_{0} $ is the concave downward and decreasing curve, in the $(\sigma, s) $ plane:
	\begin{eqnarray}
	S(\sigma) &=& \sup \{s : f \in {C}_{x_{0}}^{s,\sigma-s} \}.
	\end{eqnarray}
	\end{definition}
Under  global regularity conditions on $ f $, all the regularity exponents can be captured from the curve $ {S} (\sigma) $, obtaining a complete description of the local regularity. For example, if $S(0)>0$ and $ f $ is uniformly H\"older, i.e.,
$f$ is bounded function  and  there exists $C$ such for all $x,y$:
 
either
\begin{center}
	$\left|f(x)-f(y)\right|<C\left|x-y\right|^\varepsilon$ \;\;\;  \text{ with }$0 <\varepsilon<1$, 
\end{center}
or, if $\varepsilon>1$, $\varepsilon \in\mathbb{Z}$, there exists  all the derivates of $f$ of order less than $\varepsilon$  and  $C$ such that

\begin{center}
	$\left|\partial^{\left[\varepsilon\right]} f(x)-\partial^{\left[\varepsilon\right]} f(y)\right|<C\left|x-y\right|^{\varepsilon-\left[\varepsilon\right]},$  \;\;\; 
\end{center}

then the pointwise H\"older exponent at $x_0$ is  $S(0)$; the local H\"older exponent at $x_0$ is $ \sigma$ such that $ S (\sigma) = \sigma $; the chirp exponent at $x_0$ is the additive inverse of the slope of the asymptote to the left of $ S (\sigma) $, the oscillation exponent is the additive inverse of the slope of tangent line to the left of $ S (\sigma) $ at  $(0, S(0))$; and the weak scaling exponent at $x_0$ is   $\displaystyle{\lim_{\sigma\rightarrow -\infty} S(\sigma)}$.

It is therefore relevant to design prototype functions with a predetermined singularity structure. In this sense, in \cite {Jaffard1995, Daoudi1998}, using different methods, several functions with  prescribed pointwise H\"older exponent are constructed.
For example, given  a function $ h: [0,1] \longrightarrow [0,1] $ which is the lower limit of a sequence of continuous functions, in \cite {Daoudi1998}  different functions such that their pointwise H\"older exponent is $h (x) $ are constructed. Also, in \cite {SeuretLevy2002}, given a non-negative lower semi-continuous function $\alpha_l (x)$, the authors construct a function which has $\alpha_l (x)$ as its local H\"older exponent at all $ x\in \mathbb{R}$.
Moreover, in \cite{Jaffard2000} both, the pointwise H\"older and the chirp exponents, are prescribed. More precisely, given $ h (x) $ and $ \beta_c (x) $  bounded non-negative functions defined on $ [0,1] $, which are lower limits of continuous functions, the author provides a function  $ f $ such that $ h (x) $ and $ \beta_c (x) $ are its respective pointwise H\"older and chirp exponents at $ x \in [0,1]\setminus E $, for $ E $ a set of measure 0.

In this connection, it is natural to ask whether, given a decreasing and concave  downwards curve  $ S (\sigma) $, we can construct a function or distribution $ f $ such that the 2-microlocal frontier of $ f $ at $ x_0 $ is $ S (\sigma) $. In \cite{ Meyer1998, GuiJaffardLevy1998, LevySeuret2004} the authors address this question and in each work the authors construct, using the wavelet coefficients
of $ f $, a unique function or distribution $ f $ having the predetermined  2-microlocal frontier at $x_0$.

From these results the following questions arise:

\begin{enumerate}
	\item The distributions or functions, with prescribed 2-microlocal frontier at $ x_0 $, are different in each of the three cited works. However, is there any common characteristic in the three proposed distributions? Could all of them be defined by a generic formula?
	\item Can we characterize all functions or distributions with $ S (\sigma) $ as the 2-microlocal frontier at $ x_0 $?
	
\end{enumerate}

In this work we provide answers to these questions. For the first item, if $ S (\sigma) $  is a decreasing function    that is either concave downwards with $ S'' (\sigma) <0 $ or  linear,
we determine a generic formula which includes the distributions  proposed in \cite{Meyer1998, GuiJaffardLevy1998, LevySeuret2004} as special cases. Furthermore, this generic formula provides a prototype family of functions or distributions with a specific  singularity type at $ x_0 $, that is, with a prescribed 2-microlocal frontier at $ x_0 $. For the second question we characterize completely the functions or distributions such that $S(\sigma)$ is the 2-microlocal frontier at $x_0$ for the case 
that $S(\sigma )$ is linear.

The starting point is the generalization of the result stated in \cite {GuiJaffardLevy1998}, where a distribution $ f $ with predetermined 2-microlocal frontier is constructed. It is determined by its wavelet coefficients $ c_ {j, k} $ as follows:

\begin{proposition}{\small{\cite{GuiJaffardLevy1998}}}
	
	\label{GJLV98}
	
	Let  $S (\sigma)$ be a concave downwards and decreasing function defined on $\mathbb{R}$. Assume
that $S (\sigma)$ is not a line. Then, the wavelet coefficients
	\begin{eqnarray}
	\label{formulaGJL}
	c_{j,k}=\inf_{\sigma}\left\{2^{-j S(\sigma)}\left(1+\left|k-2^j x_0\right|\right)^{S (\sigma)-\sigma}\right\}
	\end{eqnarray}
	
	define a distribution $f$,  of which the 2-microlocal frontier at $x_0$ is $S(\sigma)$.

\end{proposition}

In this manuscript we  present a general formula which is a variation of  formula (\ref{formulaGJL}). The wavelet coefficients $c_{j,k}$ of the prototype family of distributions, with prescribed 2-microlocal frontier at $ x_0 $, will satisfy:

\begin{eqnarray}
\label{coef_resultado_ppal}
c_{j,k}\leq \mathcal{C}_{j,k}\;. \inf_{\sigma\in \mathbb{R}}  \left\{~2^{-jS(\sigma)}\left(\frac{1+\left|k-2^j x_0\right|}{\lambda_{j,k}}\right)^{S (\sigma)-\sigma}\right\},
\end{eqnarray}

with  $\mathcal{C}_{j,k} $ and   $\lambda_{j,k} $  positive sequences that satisfy some specific conditions (see Theorem \ref{teo_gral_2ml_version4}). 

Selecting $ \mathcal {C}_{j, k} $ and $\lambda_{j, k}$ appropriately, the formulas explicitly  built in \cite {GuiJaffardLevy1998}, \cite {Meyer1998} and \cite {LevySeuret2004}, can be adapted to the general formula  (\ref {coef_resultado_ppal}). Therefore the functions proposed in the three articles are members of this prototype family.

\section{A generalization of the 2-microlocal frontier prescription}

In this section we state and prove the two main results of this manuscript. Theorem \ref{teo_gral_2ml_version4} yields a wide class of functions (or distributions) whose 2-microlocal frontier is a given concave downward function $S(\sigma)$ with $S''(\sigma)<0$. This theorem contains the examples constructed by \cite{Meyer1998,GuiJaffardLevy1998,LevySeuret2004}. Theorem \ref{teo_lineal_caracterizacion} is in fact more satisfactory, since for the case that the prescribed 2-microlocal frontier is a line we characterize all functions (or distributions) whose 2-microlocal frontier is the given line.

We start with some needed preliminary results.
\subsection{Preliminary results}

We state two lemmas, without proofs, that will be useful to prove the main theorems.
Formulas (\ref{formulaGJL}) and (\ref{coef_resultado_ppal}) are based on the calculation of the infimum of the functions
$a^{S(\sigma)} b^{ S(\sigma)-\sigma},$
defined on $\mathbb{R}$, with fixed $a,b>0$.

If $S$ is the line $S(\sigma)= M\sigma+d$ with $M\leq 0$, it is easy to prove that
\begin{equation}
\label{formula_inf_lineal}
\inf_{\sigma\in \mathbb{R}}\{ a^{S(\sigma)} b^{S(\sigma)-\sigma} \}=
\left\{ 
\begin{array}{ccc}
0& \textrm{if}  &  {(ab)}^{M} \neq b\\
\\
(ab)^d & \textrm{if}  & {(ab)}^{M} = b, \\
\end{array}
\right. 
\end{equation} 

We will focus on the case
$0<a\leq 1$, since we are interested in the case $a=2^{-j}$ with $j\geq 0$.  We then have the following lemma.
\begin{lemma}	
	\label{lema_calculo_inf}
	Let $ S (\sigma) $ be a decreasing function defined on $\mathbb{R}$, such that  either $ S (\sigma) $ is concave downwards with $ S'' (\sigma) <0 $ or $ S (\sigma) $ is a line. Let $0<a\leq 1$ and $b>0$. Then  
	
	\begin{itemize}
		\item For  $a=b=1$:
		$$\inf_{\sigma\in \mathbb{R}}\{ a^{S(\sigma)} b^{S(\sigma)-\sigma} \}=1.$$
		
		\item For $1\leq b<\frac{1}{a}$:\\
		$\inf_{\sigma\in \mathbb{R}}\{ a^{S(\sigma)} b^{S(\sigma)-\sigma} \}=\left\{ 
		\begin{array}{ccc}
		a^{S(\sigma_1)} b^{S(\sigma_1)-\sigma_1}
		& \textrm{if}& \exists\; \sigma_1 :(ab)^{ S'(\sigma_1)}=b  \\
		\\

		\displaystyle{\lim_{\sigma\rightarrow +\infty} a^{S(\sigma)} b^{S(\sigma)-\sigma}}& \textrm{if}&    (ab)^{ S'(\sigma)}<b \;\forall \sigma\\
		\\
		\displaystyle{\lim_{\sigma\rightarrow -\infty} a^{S(\sigma)} b^{S(\sigma)-\sigma}}& \textrm{if}&   (ab)^{ S'(\sigma)}>b \;\forall \sigma\\
		
		\end{array}
		\right.$
		
		\item For any other $a$ and $b$    $\inf_{\sigma\in \mathbb{R}}\{ a^{S(\sigma)} b^{S(\sigma)-\sigma} \}=0.$
		
	\end{itemize}

\end{lemma}

\begin{remark}
\label{Obs_j_negativo}
 Lemma \ref {lema_calculo_inf} can be extended to the case $a>1$ if $ S (\sigma) $ is a decreasing and  concave downwards function,  with $ S'' (\sigma) <0 $. 

 The value $\sigma_1$, given in Lema  \ref{lema_calculo_inf}, is unique if $ S (\sigma) $ is a decreasing function that is concave downwards with $ S'' (\sigma) <0 $. 
\end{remark}

\begin{lemma}
	\label{lema-tecnico}
	Let $ S (\sigma) $ be a decreasing function defined on $\mathbb{R}$, such that  either $ S (\sigma) $ is concave downwards with $ S'' (\sigma) <0 $ or $ S (\sigma) $ is a line. Let $\varepsilon>0$ and  $s_0=S(\sigma_0)$. Then, there exists $B<0$ such that
	\begin{equation}\label{acotacion_cociente}
	\frac{S'(\sigma)(\sigma_0-\sigma)+S(\sigma)-s_0+\varepsilon}{S'(\sigma)-1}<B \quad \quad	\forall \sigma \in\mathbb{R}.
	\end{equation}
\end{lemma}

Inspired by the construction in \cite{Meyer1998} we write  $\mathbb{N}$ as an infinite union of disjoint infinite sets of integers:
\begin{equation}\label{particion_N}
\mathbb{N}=\;  \text{\textit{d}} \hspace{-14.5pt}\bigcup
_{m\in\mathbb{N}}{\Lambda}_m,
\end{equation}

with ${\Lambda}_m$ an infinite set of natural numbers, e.g. ${\Lambda}_m$ could be the set of natural numbers such that their binary decomposition has exactly  $m$ ones.

Let $ S (\sigma) $ be a decreasing function defined on $\mathbb{R}$, such that  either $ S (\sigma) $ is concave downwards with $ S'' (\sigma) <0 $ or $ S (\sigma) $ is a line. 
We define $\left\{r_m\right\}_{m\in\mathbb{N}}$ to be a sequence such that:
\begin{equation}
\label{rm}
\centering{
	\left\{r_m\right\}_{m\in\mathbb{N}}}
\text{ is a dense subset of } \im\left(\frac{S'(\sigma)}{S'(\sigma)-1}\right)\subseteq[0,1],
\end{equation}

or, if $S(\sigma)$ is linear,   $r_m$  is  the constant $\frac{S'(\sigma)}{S'(\sigma)-1}$ for all $m$.

The following set of subscripts will play an important role in the main theorem.
\begin{equation}\label{Ijk}
\emph{I}= \{(j,k_j): j\in \Lambda_m \text{ and } \left[ \left|k_j-2^jx_0\right| \right]=[2^{jr_m}]\},
\end{equation}

where $\left[x\right]$ denotes, as usual, the integer part of $x$.

\subsection{Main results}

\begin{theorem}
	\label{teo_gral_2ml_version4}
Let $ S (\sigma) $ be a decreasing concave downwards function defined on $\mathbb{R}$  with $S'' (\sigma) <0$.
Let $\mathcal{C}_{j,k} $ and  $\lambda_{j,k}$ be positive sequences such that for $k_j$ such that $ |k_j-2^j\;x_0|<2^j$ they verify

	\begin{enumerate}[(i)]
		
		\item For any  $C\in \mathbb{R}$, 
		$$\displaystyle {\varlimsup_{\substack{j\rightarrow +\infty}}{\left( \frac{\log_2\left(\mathcal{C}_{j,k_j}\right) }{j} +C \; \frac{\log_2\left(\lambda_{j,k_j}\right) }{j}\right) }\;\leq 0}.$$
	\end{enumerate}
	
\begin{enumerate}[(i)]
\setcounter{enumi}{1}

		\item For $(j,k_j)\in I$ given by (\ref{Ijk}),
		 \begin{center}
		 	$\displaystyle {\lim_{\substack{j\rightarrow +\infty}}\frac{\log_2\left(\mathcal{C}_{j,k_j}\right) }{j}= 0} \;\;$ \;\text{ and } \;$\displaystyle {\lim_{\substack{j\rightarrow +\infty}}\frac{\log_2\left(\lambda_{j,k_j}\right) }{j}= 0}. \;\;$
		 \end{center}

	\end{enumerate}
Let the coefficients $c_{j,k}$ be such that
	
	\begin{equation}
	\label{coef_resultado_ppal_version5}
	\left|c_{j,k}\right| \leq \mathcal{C}_{j,k}\;. \displaystyle{\inf_{\sigma\in \mathbb{R}}  \left\{~2^{-jS(\sigma)}\left(\frac{1+\left|k-2^jx_0\right|}{\lambda_{j,k}}\right)^{S (\sigma)-\sigma}\right\} },
		\end{equation}	
and if $(j,k)\in I$ we require \textbf{equality}. 

If $\psi$ is any wavelet in the Schwartz class with infinitely vanishing moments then the function (or the distribution) $f$ defined by its wavelet expansion as
	
\begin{equation}
\nonumber
f(x)=\sum_{j\geq 0}\;\sum_{\substack{k\in \mathbb{Z},\; |k-2^jx_0|<2^j}}~~c_{j,k} \; \psi (2^{j}x-k) 
\end{equation}
has  $S(\sigma)$ as its 2-microlocal frontier at $x_0$.

\end{theorem}

\begin{proof}
	Without loss of generality we consider  $x_0=0$. The set of subscripts, in formula (\ref{Ijk}), is then $$I= \{(j,k_j): j\in \Lambda_m \text{ and } \lvert k_j\rvert=[2^{jr_m}]\}.$$
	
	Let $j\geq 0 $. For simplicity we consider the equality
	$$\left|c_{j,k}\right|= \mathcal{C}_{j,k}\;. \inf_{\sigma\in \mathbb{R}}  \left\{~2^{-jS(\sigma)}\left(\frac{1+\left|k\right|}{\lambda_{j,k}}\right)^{S (\sigma)-\sigma}\right\},  $$
	for  $j\geq 0 $ and $ |k|<2^j$, although it will be clear, in the proof, that we only need the equality for $(j,k) \in I$, and the inequality if $(j,k)\notin I$.

	By using Lemma \ref{lema_calculo_inf} for $a=2^{-j}$ and $b=\frac{1+\left|k\right|}{\lambda_{j,k}}$, we compute the wavelet coefficient as:
	\begin{itemize}
		\item For  $ 2^{-j}=1$ and $\frac{1+\left|k\right|}{\lambda_{j,k}}=1$:
		$\left|c_{j,k} \right|= \mathcal{C}_{j,k}= \mathcal{C}_{0,0}$
		\item For $1\leq \frac{1+\left|k\right|}{\lambda_{j,k}}< 2^j$:
		\begin{equation}
		\label{formula_coef_dem}
		\left|c_{j,k}\right|=\left\{ 
		\begin{array}{ccc}
		
		\mathcal{C}_{j,k} \; \left( 2^{-j}\right) ^{S(\sigma_{j,k})} \left(\frac{1+\left|k\right|}{\lambda_{j,k}}\right)^{S(\sigma_{j,k})-\sigma_{j,k}} &\text{if}&  \exists \;\sigma_{j,k}: \left(2^{-j}\frac{1+\left|k\right|}{\lambda_{j,k}}\right)^{ S'(\sigma_{j,k})}=\frac{1+\left|k\right|}{\lambda_{j,k}}
		\\
		\\
		\\
		\mathcal{C}_{j,k} \;\displaystyle{\lim_{\sigma\rightarrow +\infty} (2^{-j})^{S(\sigma)} \left(\frac{1+\left|k\right|}{\lambda_{j,k}}\right)^{S(\sigma)-\sigma}}
		&\text{ if }& \left(2^{-j}\frac{1+\left|k\right|}{\lambda_{j,k}}\right)^{ S'(\sigma)}<\frac{1+\left|k\right|}{\lambda_{j,k}} \;\;\; \forall \sigma
		\\  
		\\
		\\
		\mathcal{C}_{j,k} \;\displaystyle{\lim_{\sigma\rightarrow -\infty} (2^{-j})^{S(\sigma)} \left(\frac{1+\left|k\right|}{\lambda_{j,k}}\right)^{S(\sigma)-\sigma}}
		&\text{ if }& \left(2^{-j}\frac{1+\left|k\right|}{\lambda_{j,k}}\right)^{ S'(\sigma)}>\frac{1+\left|k\right|}{\lambda_{j,k}} \;\;\; \forall \sigma
		
		\end{array}
		\right.
		\end{equation}
		
		\item For any other $j,k$:
		$$\left|c_{j,k}\right|=0$$
	\end{itemize}

	Let  $(\sigma_0, s_0)$ be  a point on the graph of $S(\sigma)$. Our purpose is to prove $$\sup \{s : f \in { C}_{0}^{s,\sigma_0-s}\}=s_0,$$
	that is to prove that \textbf{A)} $f\notin { C}_{0}^{s_0+\varepsilon, \sigma_0-(s_0+\varepsilon)}$  and   \textbf{B)} $f\in { C}_{0}^{s_0-\varepsilon, \sigma_0-(s_0-\varepsilon)}$, for all $\varepsilon>0$.
\\	
\textbf{A)} Let us show that $f\notin { C}_{0}^{s_0+\varepsilon, \sigma_0-(s_0+\varepsilon)}$ for all $\varepsilon>0$:
		
Let us assume on the contrary that there exists
$\varepsilon>0$ such that  $f\in { C}_{0}^{s_0+\varepsilon, \sigma_0-(s_0+\varepsilon)}$ which means that there exists a constant $C > 0$ such that
		\begin{eqnarray}
		\label{formula_pert}
		|c_{j,k}| \leq C 2^{-j (s_0 +\varepsilon)}\;(1 + | k  | )^{s_0 +\varepsilon-\sigma_0}
		\end{eqnarray}
		for all $ j $ and $k \in {\mathbb Z}$: $ | k| < 2^j$, and let us show that it is a contradiction.
		
We will prove that for a given  $(\sigma_0, s_0)$  it is possible to construct a sequence $(j_n, k_n)$, with $j_n$ strictly increasing, such that $$(j_n, k_n)\in \emph{I}= \{(j,k_j): j\in \Lambda_m \text{ and } \lvert k_j\rvert=[2^{jr_m}]\},$$ 
and verifies that there exists $\sigma_n$ such that
\begin{equation}\label{formula_renglon1}
		\left(2^{-j_n}\frac{1+\left|k_n\right|}{\lambda_{j_n,k_n}}\right)^{S'(\sigma_n)}=\frac{1+\left|k_n\right|}{\lambda_{j_n,k_n}} \;\text{ and }\; \lim_{n \to +\infty} \sigma_n=\sigma_0.
\end{equation}

We construct the sequence $ (j_n, k_n) $ in the following way. Since $S'(\sigma_0)<0$,  $\frac{S'(\sigma_0)}{S'(\sigma_0)-1} \in \left[0,1\right)$. Therefore, there exists a subsequence  $(r_{m_n})_{n\in\mathbb{N}}$  of  $\{r_m\}_{m\in\mathbb{N}}$ which is dense in $\im\left(\frac{S'(\sigma)}{S'(\sigma)-1}\right)$, such that 
		$$ \lim_{n \to +\infty} {r_{m_n}}= \frac{S'(\sigma_0)}{S'(\sigma_0)-1}.$$
Since the sets ${\Lambda}_{m_n}$ are infinite we can choose $j_n\in{\Lambda}_{m_n}$ a strictly increasing sequence and  $k_n=k_{j_n}=[2^{j_n r_{m_{n}}}]$. Hence $(j_n, k_n)\in I$ and
$$  2^{j_n r_{m_n}}\leq 1+ |k_n| \leq 1+2^{j_n r_{m_n}}\leq 2 \;2^{j_n r_{m_n}}.$$
Therefore, taking into account the hypothesis  $\displaystyle {\lim_{n\rightarrow +\infty}\frac{\log_2\left(\lambda_{j_n,k_n}\right) }{j_n}= 0},$ we obtain
		\begin{eqnarray}
		\nonumber
		\lim_{n \to +\infty} {\frac{\ln \left(\frac{1+\left |k_n\right|}{\lambda_{j_n,k_n}} \right) }{\ln \left(2^{-j_n}\;\frac{1+\left|k_n\right|}{\lambda_{j_n,k_n}}\right)}}  
		&=&  \lim_{n \to +\infty} {\frac{\ln \left(1+\left|k_n\right|\right)-\ln(\lambda_{j_n,k_n})  }{-j_n\ln (2)+\ln \left(1+\left|k_n\right|\right)-\ln(\lambda_{j_n,k_n}) }}
		\\ \nonumber &=& \lim_{n \to +\infty} {\frac{j_n\left( \frac{K}{j_n}+\ln (2) r_{m_n}-\frac{\ln\left(\lambda_{j_n,k_n}\right)}{j_n}\right)  }{j_n \left(\frac{K}{j_n}+\ln (2) (r_{m_n}-1)-\frac{\ln\left(\lambda_{j_n,k_n}\right)}{j_n}\right) }} \\ \nonumber &=& S'(\sigma_0).
		\end{eqnarray}
Since $S$ is strictly concave downwards, the function $S':\mathbb{R}\longrightarrow \im(S')$ is strictly decreasing and thus  bijective. Then there exists $\sigma_n $ such that 
		$$ \lim_{n \to +\infty} \sigma_n=\sigma_0   \;\; and\;\;S'(\sigma_n)=\frac{\ln \left(\frac{1+\left |k_n \right|}{\lambda_{j_n,k_n}} \right) }{\ln \left(2^{-j_n}\;\frac{1+\left|k_n\right|}{\lambda_{j_n,k_n}}\right)}$$
for $n\geq n_0$, $n_0\in \mathbb{N}$, which is  equivalent to (\ref{formula_renglon1}).
		
Furthermore, we can take $n_0$ such that, for all $n\geq n_0$,    $$1\leq \frac{1+\left|k_n\right|}{\lambda_{j_n,k_n}}\leq 2^{j_n},$$
and thus  taking $\log_2(\cdot)$ and  dividing by $j_n$, we have
		$$ 0 \leq \lim_{n \to +\infty} {\frac{\log_2 \left(\frac{1+\left |k_n \right|}{\lambda_{j_n,k_n}} \right) }{j_n}}= \lim_{n \to +\infty}{\left( \frac{K}{j_n}+ r_{m_n}-\frac{\ln\left(\lambda_{j_n,k_n}\right)}{j_n}\right)}=\frac{S'(\sigma_0)}{S'(\sigma_0)-1}<1. $$
In short, the first equality in the formula (\ref {formula_coef_dem}) holds with $\sigma_{j_n,k_n}=\sigma_n$ and therefore 
		$$|c_{j_n,k_n}|= \mathcal{C}_{j_n,k_n} \; \left( 2^{-j_n}\right) ^{S(\sigma_n)} \left(\frac{1+\left|k_n\right|}{\lambda_{j_n,k_n}}\right)^{S(\sigma_n)-\sigma_n}.$$
		
In consequence, the inequality in formula (\ref{formula_pert}) can be reformulated, for $n\geq n_0$,  $j=j_n$ and  $k=k_n$, as		
		$$ \mathcal{C}_{j_n,k_n} \; \left( 2^{-j_n}\right) ^{S(\sigma_n)} \left(\frac{1+\left|k_n\right|}{\lambda_{j_n,k_n}}\right)^{S(\sigma_n)-\sigma_n}  \leq C 2^{-j_n (s_0 +\varepsilon)}\;(1 + | k_n| )^{s_0 +\varepsilon-\sigma_0}.$$
Or equivalently,		
		$$  \mathcal{C}_{j_n,k_n} \leq C  2^{-j_n (s_0-S(\sigma_n))} 2^{-j_n\varepsilon}  \; (1 + | k_n| )^{s_0 -\sigma_0-S(\sigma_n)+\sigma_n} \; (1 + | k_n| )^\varepsilon \; {\left(\lambda_{j_n,k_n}\right)}^{S(\sigma_n)-\sigma_n}. $$		
Applying $\log_2(\cdot)$ and dividing by $j_n$  we obtain	
		\begin{eqnarray}	
		\frac{\log_2(\mathcal{C}_{j_n,k_n})}{j_n} &\leq& \frac{\log_2(C)}{j_n}  - \left(s_0-S(\sigma_n)\right) -\;\varepsilon  \;+ \nonumber
		\\
		&+& \;\frac{\log_2(1 + | k_n| )}{j_n} \left(s_0 -\sigma_0-S(\sigma_n)+\sigma_n\right) + 
		\nonumber\\ 
		&+&\; \frac{\log_2(1 + | k_n| )}{j_n}\varepsilon +\frac{ \log_2(\lambda_{j_n,k_n})}{j_n}(S(\sigma_n)-\sigma_n). \nonumber
		\end{eqnarray}
Recalling that $j_n$ and $k_n$ were selected such that $(j_n,k_n)\in I$,  that is $1+|k_n|=K 2^{j_n r_{m_n}}$ for $1\leq K\leq 2$, we obtain
\begin{eqnarray}
		\label{desig_parteA}
		\nonumber
		\frac{\log_2(C_{j_n,k_n})}{j_n} &\leq& \frac{\log_2(C)}{j_n}-\left(s_0-S(\sigma_n)\right) -\;\varepsilon+ \hspace{5cm}    \\
		 &+ &\left(\frac{\log_2(K)}{j_n} +r_{m_n}\right) \left(s_0 -\sigma_0-S(\sigma_n)  +\sigma_n\right)+
		 \nonumber
		\\
		&+&\left(\frac{\log_2(K)}{j_n} +r_{m_n} \right)\; \varepsilon +\frac{ \log_2(\lambda_{j_n,k_n})}{j_n}(S(\sigma_n)-\sigma_n).\;\;\end{eqnarray}
Therefore, by hypothesis, $\mathcal{C}_{j,k_j}$ and $\lambda_{j,k_j}$ satisfy for $(j_n,k_n)\in I$,
		\begin{center}
			$\displaystyle {\lim_{\substack{j\rightarrow +\infty\\(j,k_j)\in I}}\frac{\log_2\left(\mathcal{C}_{j,k_j}\right) }{j}= 0}$ \;  and
			$\displaystyle {\lim_{\substack{j\rightarrow +\infty\\(j,k_j)\in I}}\frac{\log_2\left(\lambda_{j,k_j}\right) }{j}= 0}.$
		\end{center}  
Taking limit for $n\rightarrow+\infty$ in (\ref{desig_parteA}), we obtain
		$$0\leq \varepsilon \left(\frac{S'(\sigma_0)}{S'(\sigma_0)-1}-1\right)=\varepsilon \; \frac{1}{S'(\sigma_0)-1}<0,$$
which is a contradiction.
\\		
\textbf{B)} Let us show that $f\in { C}_{0}^{(s_0-\varepsilon, \sigma_0-(s_0-\varepsilon))}$ for all $\varepsilon>0$:

We need to prove that there exists $C>0$ such that
		\begin{equation}\label{pertenecia-eps}
		|c_{j,k}| \leq C 2^{-j (s_0-\varepsilon)}\;(1 + | k  | )^{s_0-\varepsilon-\sigma_0} \quad \forall\;  j ,\; k \in {\mathbb Z}:\; j\geq0,\; | k| < 2^j.
		\end{equation}
In fact, it is enough to prove 	(\ref{pertenecia-eps}) for $j\geq n_0$, $n_0\in\mathbb{N}$, since $C$ can be adjusted for a finite set of $j$. In other words, it will be enough to find $n_0$ and $C$ such that
		\begin{equation}\label{pertenecia-eps1}
		\dfrac{|c_{j,k}|}{ 2^{-j (s_0-\varepsilon)}\;(1 + | k  | )^{s_0-\varepsilon-\sigma_0}}\leq C \quad \text{ for all }  j\geq n_0,\;   | k| < 2^j.
		\end{equation}
		
We use formula (\ref{formula_coef_dem}) to compute the  wavelet coefficients $c_{j,k}$ for different cases and  will show that the boundedness of  $\dfrac{|c_{j,k}|}{ 2^{-j (s_0-\varepsilon)}\;(1 + | k  | )^{s_0-\varepsilon-\sigma_0}}$ can be obtained in analogous ways for all cases.
		\begin{itemize}
			\item  For 
			\begin{eqnarray}
			\nonumber
			|c_{j,k}|=\mathcal{C}_{j,k} \; \left( 2^{-j}\right) ^{S(\sigma_{j,k})} \left(\frac{1+\left|k\right|}{\lambda_{j,k}}\right)^{S(\sigma_{j,k})-\sigma_{j,k}}, \text{ where } \sigma_{j,k} \text{ is such that } 
			\end{eqnarray}
			\begin{equation}
			\label{formula_1+k}
			\left(2^{-j}\frac{1+\left|k\right|}{\lambda_{j,k}}\right)^{ S'(\sigma_{j,k})}=\frac{1+\left|k\right|}{\lambda_{j,k}}, \text{ that is }\;\;
			1+\left|k\right|=\lambda_{j,k}\;2^{j\;\frac{S'(\sigma_{j,k})}{S'(\sigma_{j,k})-1}},
			\end{equation}
			it is enough to show that, for $j\geq n_0$ and $ | k|<2^j$,
			$$\frac{\mathcal{C}_{j,k} \; \left( 2^{-j}\right) ^{S(\sigma_{j,k})} \left(\frac{1+\left|k\right|}{\lambda_{j,k}}\right)^{S(\sigma_{j,k})-\sigma_{j,k}}}{2^{-j (s_0-\varepsilon)}\;(1 + | k| )^{s_0-\varepsilon-\sigma_0}} \text{ is bounded.}$$
		
			If we replace $1 + | k|$ by formula (\ref{formula_1+k}) and reformulate the last quotient,  we have to prove that
			\begin{equation}
			\label{formula_acotada}
			\mathcal{C}_{j,k} \; 2^{-j(S(\sigma_{j,k})-(s_0-\varepsilon))} \left(2^{j \frac{S'(\sigma_{j,k})}{S'(\sigma_{j,k})-1}}\right)^{S(\sigma_{j,k})-\sigma_{j,k}+\sigma_0-s_0+\varepsilon}\;(\lambda_{j,k})^{\sigma_0-s_0+\varepsilon}
			\end{equation}
			
			is bounded.

			\item The case
			$$|c_{j,k}|=\mathcal{C}_{j,k} \;\displaystyle{\lim_{\sigma\rightarrow +\infty} (2^{-j})^{S(\sigma)} \left(\frac{1+\left|k\right|}{\lambda_{j,k}}\right)^{S(\sigma)-\sigma}}$$
			is valid if
			$$\left(2^{-j}\frac{1+\left|k\right|}{\lambda_{j,k}}\right)^{S'(\sigma)}<\frac{1+\left|k\right|}{\lambda_{j,k}}\;\; \forall\sigma, 
			\;\;\text{ i.e.  }\;\;
			\left(\frac{1+\left|k\right|}{\lambda_{j,k}}\right)^{S'(\sigma)-1}<2^{jS'(\sigma)}\;\; \forall\sigma.$$
			 As $S'(\sigma)-1<0$, this is equivalent to
			$1+\left|k\right|> \lambda_{j,k} 2^{j\frac{S'(\sigma)}{S'(\sigma)-1}}\;\; \forall\sigma.$ 
			On the other hand, since $S(\sigma)- \sigma \longrightarrow -\infty$ when $\sigma \longrightarrow +\infty$ we can  select   $\overline{\sigma}$, sufficiently large, such that $$S(\overline{\sigma})-\overline{\sigma}-s_0+\sigma_0+\varepsilon<0.$$ 
			In particular
			\begin{eqnarray}
			\nonumber
			\left|c_{j,k}\right|=\mathcal{C}_{j,k}\;. \displaystyle{\inf_{\sigma\in \mathbb{R}}  \left\{~2^{-jS(\sigma)}\left(\frac{1+\left|k\right|}{\lambda_{j,k}}\right)^{S (\sigma)-\sigma}\right\}}\leq \mathcal{C}_{j,k} \; \left( 2^{-j}\right) ^{S(\overline{\sigma})}\left(\frac{1+\left|k\right|}{\lambda_{j,k}}\right)^{S(\overline{\sigma})-\overline{\sigma}}.
			\end{eqnarray}
			
			Therefore, to prove (\ref{pertenecia-eps1}) we only need to show  that 
		$$\frac{\mathcal{C}_{j,k} \; \left( 2^{-j}\right)^{S(\overline{\sigma})}\left(\frac{1+\left|k\right|}{\lambda_{j,k}}\right)^{S(\overline{\sigma})-\overline{\sigma}}
			}{2^{-j (s_0-\varepsilon)}\;(1 + | k| )^{s_0-\varepsilon-\sigma_0}}= \hspace{7
			cm}$$
\begin{equation}
	\label{expresion_acotada1}
	 = \mathcal{C}_{j,k} \; \left( 2^{-j}\right) ^{S(\overline{\sigma})-(s_0-\varepsilon)} \left(1+\left|k\right|\right)^{S(\overline{\sigma})-\overline{\sigma}-s_0+\sigma_0+\varepsilon} {\lambda_{j,k}}^{\overline{\sigma}-S(\overline{\sigma})},
\end{equation}
	is bounded for $j\geq n_0$ and $ | k|<2^j$.
			
			Since $ 1+\left|k\right|> \lambda_{j,k} 2^{j\frac{S'(\overline{\sigma})}{S'(\overline{\sigma})-1}}$ and $S(\overline{\sigma})-\overline{\sigma}-s_0+\sigma_0+\varepsilon<0,$ formula
	 (\ref{expresion_acotada1}) is bounded by
			\begin{equation}
			\label{formula_acotada1}
			\mathcal{C}_{j,k} \; 2^{-j(S(\overline{\sigma})-(s_0-\varepsilon))} \left(2^{j \frac{S'(\overline{\sigma})}{S'(\overline{\sigma})-1}}\right)^{S(\overline{\sigma})-\overline{\sigma}+\sigma_0-s_0+\varepsilon}\;(\lambda_{j,k})^{\sigma_0-s_0+\varepsilon}.
			\end{equation}

			\item Finally, let
			$$|c_{j,k}|=\mathcal{C}_{j,k} \;\displaystyle{\lim_{\sigma\rightarrow -\infty} (2^{-j})^{S(\sigma)} \left(\frac{1+\left|k\right|}{\lambda_{j,k}}\right)^{S(\sigma)-\sigma}} \text{ with } 1+\left|k\right|< \lambda_{j,k} 2^{j\frac{S'(\sigma)}{S'(\sigma)-1}}\;\; \forall\sigma. $$
			Similarly as before we can select $\overline{\sigma}$, negatively large,  such that $$S(\overline{\sigma})-\overline{\sigma}-s_0+\sigma_0+\varepsilon>0, \text{ and thus }$$ 
			$$\left(1+\left|k\right|\right)^{S(\overline{\sigma})-\overline{\sigma}-s_0+\sigma_0+\varepsilon}< \left(\lambda_{j,k} 2^{j\frac{S'(\overline{\sigma})}{S'(\overline{\sigma})-1}}\right)^{S(\overline{\sigma})-\overline{\sigma}-s_0+\sigma_0+\varepsilon}.$$
			
			Therefore,	$ \dfrac{|c_{j,k}|}{ 2^{-j (s_0-\varepsilon)}\;(1 + | k  | )^{s_0-\varepsilon-\sigma_0}}$ is bounded by
			
			\begin{equation}
			\label{formula_acotada2}
			\mathcal{C}_{j,k} \; 2^{-j(S(\overline{\sigma})-(s_0-\varepsilon))} \left(2^{j \frac{S'(\overline{\sigma})}{S'(\overline{\sigma})-1}}\right)^{S(\overline{\sigma})-\overline{\sigma}+\sigma_0-s_0+\varepsilon}\;(\lambda_{j,k})^{\sigma_0-s_0+\varepsilon}.
			\end{equation}

		\end{itemize}
		
		Therefore, 
		we have to show that (\ref{formula_acotada}), (\ref{formula_acotada1}) and (\ref{formula_acotada2}) are bounded. Since (\ref{formula_acotada1}) and (\ref{formula_acotada2}) are equivalent to  (\ref{formula_acotada}), with $\sigma_{j,k}= \overline{\sigma}$ for all $j,k$  because $\overline{\sigma}$ only depends on $\sigma_0$, $s_0$ and $\varepsilon$, in each case, we have to prove that, for all $j\geq n_0$ and $ | k|<2^j$, formula (\ref{formula_acotada}) or its equivalent
		$$\mathcal{C}_{j,k} \;\;(\lambda_{j,k})^{(\sigma_0-s_0+\varepsilon)}\; 2^{j\left[\frac{S'(\sigma_{j,k})(\sigma_0-\sigma_{j,k})+S(\sigma_{j,k})-s_0+\varepsilon}{S'(\sigma_{j,k})-1}\right]}, \text{ is bounded.}$$
		Taking $\log_2(\cdot)$, we obtain
		$$\log_2(\mathcal{C}_{j,k} ) +\log_2(\lambda_{j,k})\;(\sigma_0-s_0+\varepsilon)+ j\left[\frac{S'(\sigma_{j,k})(\sigma_0-\sigma_{j,k})+S(\sigma_{j,k})-s_0+\varepsilon}{S'(\sigma_{j,k})-1}\right]$$
		or,
		\begin{equation}
		\label{expresion_acotada}
		j\left[\frac{\log_2(\mathcal{C}_{j,k})}{j} +\frac{\log_2(\lambda_{j,k})}{j}\;(\sigma_0-s_0+\varepsilon)+\frac{S'(\sigma_{j,k})(\sigma_0-\sigma_{j,k})+S(\sigma_{j,k})-s_0+\varepsilon}{S'(\sigma_{j,k})-1}\right].
		\end{equation}
		
Since $|k|<2^j$, by hypothesis 
		$$\displaystyle {\varlimsup_{\substack{j\rightarrow +\infty}}{\left(\frac{\log_2\left(\mathcal{C}_{j,k}\right) }{j} +(\sigma_0-s_0+\varepsilon) \; \frac{\log_2\left(\lambda_{j,k}\right) }{j}\right)}\;\leq 0}.$$ 
		Further,  Lema \ref{lema-tecnico}  gives 
		$$\frac{S'(\sigma_{j,k})(\sigma_0-\sigma_{j,k})+S(\sigma_{j,k})-s_0+\varepsilon}{S'(\sigma_{j,k})-1}<B <0	\;\;\text{ for all } \sigma_{j,k}.  $$
		
Hence,
		$$\varlimsup_{\substack{j\rightarrow +\infty}}\log_2\left(\mathcal{C}_{j,k} \;\;(\lambda_{j,k})^{(\sigma_0-s_0+\varepsilon)}\; 2^{j\left[\frac{S'(\sigma_{j,k})(\sigma_0-\sigma_{j,k})+S(\sigma_{j,k})-s_0+\varepsilon}{S'(\sigma_{j,k})-1}\right]}\right) =-\infty$$		
which implies that there exists $M>0$ such that $$0<\mathcal{C}_{j,k} \;(\lambda_{j,k})^{(\sigma_0-s_0+\varepsilon)}\; 2^{j\left[\frac{S'(\sigma_{j,k})(\sigma_0-\sigma_{j,k})+S(\sigma_{j,k})-s_0+\varepsilon}{S'(\sigma_{j,k})-1}\right]}<M,$$
for $ j $ and $k \in {\mathbb Z}$ such that, $j\geq 0, \; | k| < 2^j$.

\end{proof}

In Theorem \ref{teo_gral_2ml_version4} we state sufficient conditions to define a funtion (or a distribution) $f$ that has $S(\sigma)$ as its 2-microlocal frontier at $x_0$, for a downward concave function with $S''(\sigma)<0$. Note that the requirement that $S''(\sigma)<0$ excludes the possibility that $S(\sigma)$ is a line.

For the case that $S(\sigma)$ is a line we have a more general theorem, giving necessary and sufficient conditions for a function (or distribution) to have that line as its prescribed 2-microlocal frontier.

\begin{theorem}
	\label{teo_lineal_caracterizacion}
	Let $S(\sigma )$ be the line $S(\sigma )= \alpha + \frac{\gamma}{1-\gamma} (\alpha - \sigma)$, with $0 \leq \gamma < 1$.
	Let $f$ be the function (or the distribution)  defined by its wavelet expansion as
	
	\begin{equation}
	\nonumber
	f(x)=\sum_{j\geq 0}\;\sum_{\substack{k\in \mathbb{Z},\; |k-2^jx_0|<2^j}}~~c_{j,k} \; \psi (2^{j}x-k), 
	\end{equation}
where $\psi$ is any wavelet in the Schwartz class with infinitely vanishing moments.

The 2-microlocal frontier of $f$ at $x_0$ is $S(\sigma)$  \textbf{if and only if} for $j\geq0$ and $|k-2^jx_0|<2^j$,
	\begin{equation}
	\label{coef_lineal_generalizacion}
	|c_{j,k}|\leq \mathcal{C}_{j,k}\;. \displaystyle{\inf_{\sigma\in \mathbb{R}}  \left\{~2^{-jS(\sigma)}\left(\frac{1+\left|k-2^jx_0\right|}{\lambda_{j,k}}\right)^{S (\sigma)-\sigma}\right\} }	
\end{equation}
with $\mathcal{C}_{j,k} $ and  $\lambda_{j,k} $  positive sequences such that: 
	\begin{enumerate}[(i)]
	\item For any  $C\in\mathbb{R}$ 
		$$\displaystyle {\varlimsup_{\substack{j\rightarrow +\infty}}{\left(\frac{\log_2\left(\mathcal{C}_{j,k_j}\right) }{j} +C \; \frac{\log_2\left(\lambda_{j,k_j}\right) }{j}\right)}\;\leq 0}.$$
		
	\item There exists a sequence $(j_n, k_n)$,  with $j_n$ strictly increasing, such that for $(j_n,k_n)$ the \textbf{equality}  holds in  (\ref{coef_lineal_generalizacion}), with	$c_{j_n,k_n}\neq0$, and 
	$$\displaystyle {\lim_{\substack{n\rightarrow +\infty}}\frac{\log_2\left(\mathcal{C}_{j_n,k_{n}}\right) }{j_n}= 0}  \;\;\text{  and  } \;\;\displaystyle {\lim_{\substack{n\rightarrow +\infty}}\frac{\log_2\left(\lambda_{j_n,k_{n}}\right) }{j_n}= 0}.$$
	\end{enumerate}
	
\end{theorem}

\begin{remark}
\label{Obs_coef_lineal_generalizacion}
Note that the sufficiency in this last Theorem includes a special case of Theorem 2.1. if in that Theorem we allowed $S(\sigma)$ to be a line. In that Theorem, 
there are always infinite $c_{jk} \not= 0$. But for the case of a line we have to explicitly require it.

\end{remark}
\begin{proof}
Without loss of generality we consider  $x_0=0$.
\\
$\Leftarrow )$ The proof is similar and even simpler than the one  given in  Theorem \ref{teo_gral_2ml_version4}. 
Since we have the hypothesis that $c_{j_n,k_n}\neq 0$ and  $$\left| c_{j_n,k_n}\right| =\displaystyle{\inf_{\sigma\in \mathbb{R}}  \left\{~2^{-j_nS(\sigma)}\left(\frac{1+\left|k_n\right|}{\lambda_{j_n,k_n}}\right)^{S (\sigma)-\sigma}\right\}},$$  
by formula (\ref{formula_inf_lineal}), we have
$|c_{j_n,k_n}| = \mathcal{C}_{j_n,k_n}\; 2^{-j\alpha}$ with  $ 1+|k_n| = \lambda_{j_n,k_n} 2^{j_n \gamma}.$ The equality
$1+|k_n| = \lambda_{j_n,k_n} 2^{j_n \gamma}$ is equivalent to
$$\left(2^{-j_n}\frac{1+\left|k_n\right|}{\lambda_{j_n,k_n}}\right)^{S'(\sigma_0)}=\frac{1+\left|k_n\right|}{\lambda_{j_n,k_n}},$$
and thus formula (\ref{formula_renglon1}), in the proof of Theorem \ref{teo_gral_2ml_version4},  is satisfied for $\sigma_n= \sigma_0$.

Also in this case the arguments given in Theorem \ref{teo_gral_2ml_version4} for the set $I$, defined in formula (\ref{Ijk}), can be adapted to the set  $\left\lbrace (j_n, k_n)\right\rbrace_{n\in\mathbb{N}}$ of (\textit{ii}).

$\Rightarrow$) For the necessity we first have to prove that the wavelet coefficients of the function $f$ can be described as in (\ref{coef_lineal_generalizacion}). For simplicity we consider the equality, that is to reformulate the given $| c_{j,k}|$ as 
\begin{equation}
\nonumber
\mathcal{C}_{j,k}. \inf_{\sigma\in \mathbb{R}}  \left\{~2^{-jS(\sigma)}\left(\frac{1+|k|}{\lambda_{j,k}}\right)^{S (\sigma)-\sigma}\right\} =
\left\{
	\begin{array}{ccc} \mathcal{C}_{j,k} \;2^{-j \alpha}  &\text{ for }&    1+|k| = \lambda_{j,k} 2^{j \gamma}
	 \\\\
	0  &\text{ for }&   1+|k| \neq \lambda_{j,k} 2^{j \gamma}.
\end{array}
\right.
\end{equation}
For  $|c_{j,k}|\neq 0$, it is enough to select
\begin{equation}
\label{formula_lambda_y_Cal_nonulos}
\lambda_{j,k}=\frac{1+|k|}{2^{j\gamma}} \quad \text{ and }\quad \mathcal{C}_{j,k}=\frac{|c_{j,k}|}{2^{-j\gamma}}.
\end{equation}

And, for $|c_{j,k}|=0$ it is sufficient to choose
$\lambda_{j,k}\neq\frac{1+|k|}{2^{j\gamma}} $ to be zero in  (\ref{coef_lineal_generalizacion}), e.g. 
\begin{equation}
\label{cjk_nulos}
\lambda_{j,k}=2^{j} \quad \text{ and }\quad \mathcal{C}_{j,k}=2^{-j^2}.
\end{equation}

Let us prove that 	$\mathcal{C}_{j,k} $ and  $\lambda_{j,k} $ satisfy conditions (\textit{i}) and (\textit{ii}).
\\\\
(\textit{i})  For   $c_{j,k}= 0$ condition (\textit{i}) is trivial. 
	
For  $c_{j,k}\neq 0$, since the 2-microlocal frontier of the function $f$ is $S(\sigma)$ we have $S(\sigma_0)= \sup \{s : f \in {C}_{x_{0}}^{s,\sigma_0-s}\}$ for all $\sigma_0\in \mathbb{R}$.  This means that for all $\varepsilon>0$ and  $(\sigma_0, s_0)\in \text{Graph}(S)$,
	\begin{center}
		1) 	$f\notin { C}_{0}^{s_0+\varepsilon, \sigma_0-(s_0+\varepsilon)}$ \quad and \quad 2)	$f\in { C}_{0}^{s_0-\varepsilon, \sigma_0-(s_0-\varepsilon)}$.
	\end{center}
By 2),  for  $j\geq0$, $|k|<2^j$, there exists $C>0$ such that
	$$\frac{|c_{j,k|} 
	}{2^{-j (s_0-\varepsilon)}\;(1 + | k| )^{s_0-\varepsilon-\sigma_0}}\leq C.$$
Thus, by the definitions of (\ref{formula_lambda_y_Cal_nonulos}),
\begin{equation}\label{formula2cjk}
\frac{\mathcal{C}_{j,k}2^{-j\alpha}} 
	{2^{-j (s_0-\varepsilon)}\;{(\lambda_{j,k}2^{j\gamma})} ^{s_0-\varepsilon-\sigma_0}}\leq C.
\end{equation}
Since $s_0= S(\sigma_0 )= \alpha + \frac{\gamma}{1-\gamma} (\alpha - \sigma_0)$ i.e. $(1-\gamma)s_0-\alpha+\gamma \sigma_0=0$, (\ref{formula2cjk}) can be reformulated as 
	$$\mathcal{C}_{j,k}\;{\lambda_{j,k}} ^{-s_0+\varepsilon+\sigma_0}\;\; 2^{j\varepsilon(\gamma-1)}
	\leq C.$$
Taking $\log_2(\cdot)$ and dividing by $j$, we obtain
	$$\frac{\log_2(\mathcal{C}_{j,k})}{j}+\frac{\log_2(\lambda_{j,k})}{j} (\sigma_0-s_0+\varepsilon)+ \varepsilon(\gamma-1)
	\leq \frac{\log_2(C)}{j}.$$
And thus, taking  $\varepsilon\longrightarrow0$ and letting $j\longrightarrow +\infty$, we have
	$$\displaystyle {\varlimsup_{\substack{j\rightarrow +\infty}}{\left(\frac{\log_2\left(\mathcal{C}_{j,k_j}\right) }{j} +  \; (\sigma_0-s_0)\frac{\log_2\left(\lambda_{j,k_j}\right) }{j}\right)}\;\leq 0},$$
and since $\sigma_0-s_0$ can take any real value, we obtain condition (\textit{i}).
\\\\
(\textit{ii}) Let $J:=\{(j,k):1+|k| = \lambda_{j,k} 2^{j \gamma}, j\geq 0, |k|<2^j \}$. 

By (\textit{i}), taking $C=0$, we have
	$$\displaystyle {\varlimsup_{\substack{j\rightarrow +\infty}}{\frac{\log_2\left(\mathcal{C}_{j,k_j}\right) }{j}} \leq 0},\;\; \text{ for } (j,k_j)\in J.$$
	
We need to construct a sequence $(j_n,k_n) \in J$ such that $j_n$ is strictly increasing and
$$\displaystyle {\lim_{\substack{n\rightarrow +\infty}}\frac{\log_2\left(\mathcal{C}_{j_n,k_{n}}\right) }{j_n}= 0}  \;\;\text{  and  } \;\;\displaystyle {\lim_{\substack{n\rightarrow +\infty}}\frac{\log_2\left(\lambda_{j_n,k_{n}}\right) }{j_n}= 0}.$$ 
For this we will show that there is a sequence $(j,k_j)\in J$ such that,
$$\displaystyle {\varlimsup_{\substack{j\rightarrow +\infty}}{\frac{\log_2\left(\mathcal{C}_{j,k_j}\right) }{j} }}=0.$$

For, assume that for all $(k_j)_{j\in\mathbb{N}}$ such that  $(j,k_j)\in J$ there exists $\delta<0 $ such that
$$\displaystyle {\varlimsup_{\substack{j\rightarrow +\infty}}{\frac{\log_2\left(\mathcal{C}_{j,k_j}\right) }{j} }}= \delta\left( (k_j)_{j\in\mathbb{N}}\right)<\delta<0.$$
	
Taking $\varepsilon>0$ such that $$\delta<\varepsilon (\gamma-1)<0 \quad i.e.\quad \varepsilon< \dfrac{\delta}{\gamma-1},$$
we have 
	\begin{equation}\label{calC_jk_desig}
	\frac{\log_2\left(\mathcal{C}_{j,k_j}\right) }{j}< \delta< \varepsilon (\gamma-1)\quad i.e. \quad \mathcal{C}_{j,k_j} < 2^{j \varepsilon (\gamma-1)},
	\end{equation}
for all $j\geq j_0$ and   all $(k_j)_{j\in\mathbb{N}}$ such that $(j,k_j)\in J$.
	
Since $(1-\gamma)s_0=\alpha-\gamma \sigma_0$ we have $$\varepsilon (\gamma-1)=\alpha-(s_0+\varepsilon)+\gamma(s_0+\varepsilon-\sigma_0),$$ and   (\ref{calC_jk_desig}) can be reformulated as
	\begin{equation}\label{calC_desigualdad}
\mathcal{C}_{j,k_j}2^{-j\alpha} <  \;2^{-j(s_0+\varepsilon)}\;(2^{j\gamma})^{s_0+\varepsilon-\sigma_0}. 
	\end{equation}
	
Hence, for $ c_{j,k}\neq0$ i.e.  $|c_{j,k}| = \mathcal{C}_{j,k} \;2^{-j \alpha}$ and $ 1+|k_j| = \lambda_{j,k_j} 2^{j \gamma}$, formula (\ref{calC_desigualdad}) is equivalent to
	$$|c_{j,k_j}|<2^{-j(s_0+\varepsilon)}\;{\left( \frac{1+\left|k_j\right|}{\lambda_{j,k_j}}\right)}^{ s_0+\varepsilon-\sigma_0}= 2^{-j(s_0+\varepsilon)}\;{\left( 1+\left|k_j\right|\right)}^{ s_0+\varepsilon-\sigma_0} \lambda_{j,k_j}^{ \sigma_0-s_0-\varepsilon}.$$
Thereby, for $\varepsilon< \dfrac{\delta}{\gamma-1}$  and $(\sigma_0,s_0)$ such that $\sigma_0-s_0-\varepsilon=0$ we have 
	$$|c_{j,k_j}|< 2^{-j(s_0+\varepsilon)}\;{\left( 1+\left|k_j\right|\right)}^{ s_0+\varepsilon-\sigma_0}, $$
for all $j\geq j_0$ and all sequences  $(k_j)_{j\in\mathbb{N}}$,  $(j,k_j)\in J$.
	
In the case $ c_{j,k}=0$ the last inequality is obvious. And thus, there exists $C>0$ such that
	$$|c_{j,k}| \leq C 2^{-j (s_0 +\varepsilon)}\;(1 + | k  | )^{s_0 +\varepsilon-\sigma_0},$$
for all $ j\geq 0 $ and $k \in {\mathbb Z}$: $ | k| < 2^j$, which means that  $f\in { C}_{0}^{s_0+\varepsilon, \sigma_0-(s_0+\varepsilon)}$. This fact contradicts the hypothesis stated in 1) and the  contradiction arises because we assumed that 
$\displaystyle {\varlimsup_{\substack{j\rightarrow +\infty}}{\frac{\log_2\left(\mathcal{C}_{j,k_j}\right) }{j} }}= \delta\left( (k_j)_{j\in\mathbb{N}}\right)<\delta<0,$
for all $(k_j)_{j\in\mathbb{N}}$,  $(j,k_j)\in J$. Consequently, we can choose $(j_m,k_m)\in J$ such that 
	$$\displaystyle {\lim_{\substack{m\rightarrow +\infty}}\frac{\log_2\left(\mathcal{C}_{j_m,k_{m}}\right) }{j_m}= 0}.$$
Furthermore, since we have 
	$$\displaystyle {\varlimsup_{\substack{j\rightarrow +\infty}}{\left(\frac{\log_2\left(\mathcal{C}_{j,k_j}\right) }{j} +  \; C\;\frac{\log_2\left(\lambda_{j,k_j}\right) }{j}\right)}\;\leq 0},$$ 
for all $(k_j)_{j\in\mathbb{N}}$,  $(j,k_j)\in J$ and for all $C$,
we obtain
	$$\displaystyle {\varlimsup_{\substack{m\rightarrow +\infty}}\left(\frac{\log_2\left(\mathcal{C}_{j_m,k_{m}}\right) }{j_m}+ C\;\frac{\log_2\left(\lambda_{j_m,k_{m}}\right) }{j_m}\right)\leq 0} \quad \text{ for all } C.$$

Thus,  $$\displaystyle {\varliminf_{\substack{m\rightarrow +\infty}} C\;\frac{\log_2\left(\lambda_{j_m,k_{m}}\right) }{j_m}\;\leq 0} \quad \text{ for all } C,$$
which is only valid if
	$\displaystyle {\varliminf_{\substack{m\rightarrow +\infty}}\;\frac{\log_2\left(\lambda_{j_m,k_{m}}\right) }{j_m}= 0}.$  Hence, there exists a subsequence  $(j_{m_n}, k_{m_n})\in J$ such that 
	$${\lim_{\substack{n\rightarrow +\infty}}\;\frac{\log_2\left(\lambda_{j_{m_n},k_{m_n}}\right) }{j_{m_n}}= 0}.$$
Therefore, choosing $(j_n,k_n)=(j_{m_n},k_{m_n})$ both conditions of (\textit{ii}) are satisfied.
			
 \end{proof}

\subsection{ Connecting the main theorem with previous results about 2-microlocal frontier prescription}
Theorem \ref {teo_gral_2ml_version4} generalizes Proposition \ref {GJLV98} since if we select $\mathcal{C}_{j,k}= 1 $,  $\;\lambda_{j,k}=1 $ and in formula (\ref{coef_resultado_ppal_version5}) we only consider equality, we obtain that proposition. 

We now consider the results  proposed in \cite{Meyer1998} and \cite{LevySeuret2004}. We will  examine how they can be adapted to the formula stated in Theorem \ref{teo_gral_2ml_version4}.

Y. Meyer develops his result in the  $(s, s')$ plane and  considers the 2-microlocal as the set of $(s, s')$ such that $ s = A (s') $.
In \cite{Meyer1998} the following theorem is proved:
\begin{theorem} {\small{(Meyer, 1998)}}
	Let $A:\mathbb{R}\longrightarrow
	\mathbb{R}$ be a concave downwards Lipschitz function which is decreasing on the real line with   $-1\leq \frac{dA}{dt}(t)\leq 0$. Let $E$ be the planar set defined by $s\leq A(s')$, $(s,s')\in\mathbb{R}^2$.
	Then there exists a function $f$ defined on a neighbourhood of $x_0$ such that the 2-microlocal domain of $f$ at $x_0$ (see formula (\ref{dom2ml}) ) is $E$. 
	
\end{theorem}
He constructs a function $f$ 
whose  the wavelet coefficients satisfy
$$ c_{j,k}= 2^{-j\tau_m}~~~if~~~j\in{\Lambda}_m ~and~  k_j=\left[2^{jp_m}\right]  $$
and zero otherwise, where ${\Lambda}_m$ is defined as in (\ref{particion_N}) (see \cite{Meyer1998} for more details). The  2-microlocal domain of $f$ at $x_0=0$ is $E$, and therefore  $s= A(s')$ is the 2-microlocal frontier of $f$ at $0$. 

In our setting in the  $(\sigma,s)$ plane, the 2-microlocal frontier can be written as  
\begin{center}
	$S(\sigma)=s$  such that $s=A(\sigma-s)$.
\end{center}

Using the formula of Theorem \ref{teo_gral_2ml_version4}, by selecting  $\mathcal{C}_{j,k}=1$ and $\lambda_{j,k}$ such that

\begin{equation}
\nonumber
\lambda_{j,k}=
\left\{ 
\begin{array}{ccc}
\frac{1+\left|k\right|}{2^{jp_m}}&\textrm{ if }&  j\in{\Lambda}_m\; and \; k=\left[2^{jp_m}\right] \\
	\\
1&\;\;\;\;\;\;\;\;\textrm{ otherwise, }&
\end{array}
\right.
\end{equation}
we have that for  
 $j\in{\Lambda}_m$  and   $k=\left[2^{jp_m}\right]$,
$$c_{j,k}=\mathcal{C}_{j,k}\;. \inf_{\sigma\in \mathbb{R}}  \left\{~2^{-jS(\sigma)}\left(\frac{1+\left|k\right|}{\lambda_{j,k}}\right)^{S(\sigma)-\sigma}\right\}=2^{-j\tau_m}.
$$
And, for any other $j,k$, we choose
$$c_{j,k}=0\leq \mathcal{C}_{j,k}\;. \inf_{\sigma\in \mathbb{R}}  \left\{~2^{-jS(\sigma)}\left(\frac{1+\left|k\right|}{\lambda_{j,k}}\right)^{S (\sigma)-\sigma}\right\}.$$

Therefore the function $f$ proposed by Y. Meyer satisfies the conditions of Theorem  \ref{teo_gral_2ml_version4}. 

Finally, J. L\'evy Véhel and S. Seuret work in the $(s', \sigma)$ plane and thus the  2-microlocal frontier is an increasing function defined by  $\sigma=g(s')$. In  \cite{LevySeuret2004} they present the following theorem:

\begin{theorem}{\small{(J. L\'evy Véhel and S. Seuret, 2004)}}
	
	Let $g:\mathbb{R}\longrightarrow \mathbb{R}$ be a concave downwards, non-decreasing function, with slope between 0 and 1. Assume that $g(0) > 0$. There exists a function $f$ such that the 2-microlocal frontier of $ f$ at $0$
is $g(s')$.
\end{theorem}

The construction of the function $f$ on  $[0,1]$ is as follows:

\begin{equation}\label{LevySeuret_coef_wav}
c_{j,k}= 2^{-j \beta_{j,k}} \quad \text{ for } j>0 \text{ and } k\geq 0,
\end{equation}

with $$ \beta_{j,k}=\inf_{\rho \in E_{j,k}}{\chi_0^j(\rho)} \quad\text{ and } \quad \chi_0^j(\rho)= \min\{j, -g^*(\rho)\},$$ where $g^*$ is the Legendre transform of $g$, and 
$$E_{j,k}=\{\rho: 0\leq \rho\leq 1 ~y~k= [ 2^{j(1-\rho)}] \}. $$

If we select appropriate sequences $\mathcal{C}_{j,k} $ and  $\lambda_{j,k}$  we can also describe the wavelet coefficients defined in formula (\ref{LevySeuret_coef_wav}), via the  formula  (\ref{coef_resultado_ppal_version5}) stated in Theorem \ref{teo_gral_2ml_version4}, if the function  $g$ satisfies the additional hypothesis of our theorem. We characterize  $\sigma=g(s')$,  in the  $(\sigma,s)$ plane, by defining  $S(\sigma)$ as $S(\sigma)=s$ such that $\sigma=g(\sigma-s).$

In this case, by choosing:

\begin{equation}\label{mathcalCjk}
\mathcal {C}_{j,k}=\left\{ 
\begin{array}{ccc}
1 & \textrm{ if }&j> -g^*(\rho_{j,k}) \\
\\

 2^{-j g^*(\rho_{j,k})}2^{-j^2} & \textrm{ if }&  j\leq -g^*(\rho_{j,k}) \text{ and }-g^*(\rho_{j,k}) \text{ is finite}
 \\
 \\
  2^{-j g^*(g'(0))}2^{-j^2}
 & \textrm{ if }&  j\leq -g^*(\rho_{j,k}) \text{ and }-g^*(\rho_{j,k})=+\infty,
\end{array}
\right.
\end{equation}

and

\begin{equation}\label{lambdajk}
\lambda_{j,k}=\left\{ 
\begin{array}{ccc}
\frac{1+|k|}{2^{j(1-\rho_{j,k})}} & \textrm{ if }&-g^*(\rho_{j,k}) \text{ is finite}
\\
\\
\frac{1+|k|}{2^{j(1-g'(0))}}
& \textrm{ if }& -g^*(\rho_{j,k})=+\infty,
\end{array}
\right.
\end{equation}

we have that $$c_{j,k}= 2^{-j \beta_{j,k}}=\mathcal{C}_{j,k}\;. \inf_{\sigma\in \mathbb{R}}  \left\{~2^{-jS(\sigma)}\left(\frac{1+\left|k\right|}{\lambda_{j,k}}\right)^{S(\sigma)-\sigma}\right\}, $$
for $j>0$ and $k\geq0$.
For $j=0$ or $k<0$ it  is sufficient to define
$c_{j,k}=0$ to verify the inequality  in (\ref{coef_resultado_ppal_version5}).

\section{Conclusions}

In this work we present a generic formula, based on wavelet coefficients, that provides a wide class of functions or distributions that have $S(\sigma)$ as the 2-microlocal frontier at $x_0$, where  $ S (\sigma) $ is a decreasing function defined on $\mathbb{R}$, such that  either $ S (\sigma) $ is concave downwards with $ S'' (\sigma) <0 $ or $ S (\sigma) $ is a line.
Each one of the functions constructed in \cite{ Meyer1998, GuiJaffardLevy1998, LevySeuret2004}  belongs to the family of functions we construct. The results could be useful to provide a model for signals that have  a specific type of pointwise singularity. 

Moreover, if $S(\sigma)$ is a line we prove that our condition is also necessary. In fact we conjecture that for $S(\sigma)$ a general downward concave function our sufficient conditions should be very close to be necessary.

Our theorems could be a starting point to determine which are the compatibility conditions satisfied by the different 2-microlocal frontiers of a function or distribution when the singularities at $x$ varies in $I$ an interval. This issue is still  unresolved and there is only a simultaneous prescription in \cite{LevySeuret2004}, on a countable dense set of points. Therefore, some questions arise:  Is it possible to extend our result to an interval $I$? What conditions have to be required on a set of curves to be able to define a function $f:I \longrightarrow \mathbb{R}$ having this set of curves as its 2-microlocal frontiers in $I$?

\section*{References}

\bibliographystyle{elsarticle-num}

\bibliography{bibliografia_art} 
\end{document}